\documentclass[11pt]{article}
\usepackage{amssymb}
\usepackage{graphicx}
\usepackage{enumerate}
\usepackage{multirow}
\usepackage{amsmath}

\newtheorem{proposition}{Proposition}[section]

\newtheorem{lemma}{Lemma}[section]
\newtheorem{theorem}{Theorem}[section]
\newtheorem{corollary}{Corollary}[section]

\newenvironment{proof}
{% This is the begin code
\textbf{Proof.}
}
{% This is the end code
\hfill$\square$}

\def\R{\mbox{$\mathbb{R}$}}

\def\x{{\mbox{\boldmath $x$}}}

\def\sm{{\hspace*{-.5pt}{\setminus}\hspace*{-.5pt}}}

\def\vec0{{\mbox{\boldmath $0$}}}

\begin{document}

\title{Spectral characterizations of almost complete graphs}
\author{Marc C\'amara and Willem H. Haemers\thanks{corresponding author, e-mail: {\tt haemers@uvt.nl}}
\\[5pt]
{\small Tilburg University, The Netherlands}}
\date{}
\maketitle
\abstract{\noindent
We investigate when a complete graph $K_n$ with some edges deleted is determined by its adjacency spectrum.
It is shown to be the case if the deleted edges form a matching, a complete graph $K_m$ provided
$m \leq n-2$, or a complete bipartite graph.
If the edges of a path are deleted we prove that the graph is determined by its generalized spectrum
(that is, the spectrum together with the spectrum of the complement).
When at most five edges are deleted from $K_n$, there is just one pair of nonisomorphic cospectral graphs.
We construct nonisomorphic cospectral graphs (with cospectral complements) for all $n$ if six or more edges
are deleted from $K_n$, provided $n$ is big enough.
}
%%%%%%%%%%%%%%%%%%%%%%%%%%%%%%%%%%%%%%%%%%%
\section{Introduction}
%%%%%%%%%%%%%%%%%%%%%%%%%%%%%%%%%%%%%%%%%%%

Two graphs for which the adjacency matrices have the same spectrum are called {\em cospectral}.
A graph $G$ is {\em determined by its spectrum} ({\em DS} for short) if every graph cospectral with $G$
is isomorphic with $G$.
Spectral characterizations of graphs (with respect to various matrices) did get much attention
in the recent past; see \cite{DH03,DH09}.
It has been conjectured by the second author that almost all graphs are DS.
Truth of this conjecture would mean that the spectrum gives a useful fingerprint for a graph.
The paradox is that it is difficult to prove that a given graph is DS.
Not very many classes of graphs are known to be DS.
These include for example the path $P_n$ the cycle $C_n$ and the complete graph $K_n$.
A number of papers have appeared that prove spectral characterizations for more complicated cases.
Very often such graphs have relatively few edges, like T-shape trees and lollipop graphs (see \cite{WX06,BJ08}).
Not many results are known if $G$ has many edges, that is, the complement of $G$ has few edges.
If $G$ is regular, or if one considers the spectrum of the Laplacian matrix,
then a graph is DS if and only if the complement is.
However, with respect to the adjacency spectrum of a nonregular graph $G$ with few edges,
the characterization problem for the complement of $G$, is most of the time much harder than for $G$.
For example for the path $P_n$, there is a straightforward proof that $P_n$ is DS (see for example \cite{DH03}).
However, for the complement of $P_n$ the proof is rather involved (see \cite{DH02}).

If $H$ is a subgraph of a graph $G$, then the graph obtained from $G$ by deleting the edges of $H$ is
denoted by $G { \sm } H$.
In this paper the following graphs are proved to be DS:
$K_n{ \sm } \ell K_2$, $K_n{ \sm } K_m$ (provided $m\leq n-2$),
$K_n { \sm } K_{\ell,m}$ and $K_n { \sm } G$, when $G$ has at most four edges.
We show that there is exactly one pair of nonisomorphic cospectral graphs if five edges are deleted from $K_n$.
If six or more edges are deleted from $K_n$, one can obtain cospectral graphs for
every $n$ which is big enough.

The graph $G=P_\ell+(n-\ell)K_1$
%is DS if $\ell$ is even and
has a nonisomorphic cospectral mate if $\ell$ is odd and $5\leq\ell\leq n-1$ (see \cite{CSS10}).
The complements of these cospectral graphs are not cospectral.
%For the complement $K_n{ \sm } P_\ell$ the spectral characterization turns out to be harder.
More generally, we prove that $G$ is determined by the generalized spectrum (which is the spectrum
of $G$ together with the spectrum of the complement).

\section{Removing a matching or a complete graph}

It is known that the adjacency spectrum determines the number of closed walks of any given
length $\ell$.
% ($\ell$-walks from now on).
For $\ell=0$, $2$ and $3$ this implies that cospectral graphs have the same number of
vertices, edges and triangles, respectively.
Deleting one edge from $K_n$ destroys $n-2$ triangles, and deleting $m$ edges destroys
at most $m(n-2)$ triangles, with equality if and only if the deleted edges form a matching.
Therefore any graph with $n$ vertices, ${n \choose 2}-m$ edges and ${m\choose 3}-m(n-2)$ triangles
is $K_n{ \sm } mK_2$.
Thus we can conclude:
\begin{proposition}\label{match}
A graph obtained from $K_n$ by removing the edges of a matching is DS.
\end{proposition}
Suppose $G = K_n { \sm } K_m$ .
Then $G$ has adjacency matrix
\[
A=\left[\begin{array}{cc} O_m & J \\ J & J-I_{n-m}\end{array} \right]
\]
(as usual, $O$, $J$ and $I$ are the all-zero, all-one and identity matrix, respectively; indices indicate the order).
We see that rank$\, A = n-m+1$ and rank$(A+I) = m+1$, hence $A$ has an eigenvalue $0$
with multiplicity $m-1$ and an eigenvalue $-1$ with multiplicity $n-m-1$.
%$\{ 0^{m-1}, -1^{n-m-1} , \frac{1}{2}(n-m-1\pm\sqrt{n^2-3m^2-2mn+2n+2m+1}) \}$.
We shall prove that $G$ is DS provided $m\leq n-2$.
The first step in the proof seems interesting in its own right.
\begin{lemma}
If a graph $G$ has only one positive eigenvalue, then $G$ is a complete
multipartite graph, possibly extended with some isolated vertices.
\end{lemma}
\begin{proof}
Suppose $G$ has $F=K_2+K_1$ as an induced subgraph.
Let $x$ be the isolated vertex in $F$.
Assume $x$ is not isolated in $G$, then $x$ is adjacent to some vertex $y$ of $G$ outside $F$.
The vertices of $F$ together with $y$ induce a subgraph of $G$ on four vertices
containing two disjoint edges.
There are just three such graphs: $2K_2$, $P_4$, and a triangle with one pendant edge.
All three have a positive second eigenvalue which contradicts the interlacing inequalities.
Therefore $F$ is not an induced subgraph of $G$, and therefore any two nonadjacent vertices of
$G$ have the same neighbors, which proves the claim.
\end{proof}
\begin{theorem}\label{complete}
If $m\leq n-2$, then $K_n { \sm } K_m$ is DS.
\end{theorem}
\begin{proof}
Let $G$ be a graph cospectral with $K_n{ \sm } K_m$.
By the above lemma, $G$ consist of a complete multipartite graph $G'$ and possibly some isolated vertices.
The two smallest eigenvalues of the complete tripartite graph $K_{2,2,1}$ are $-2$ and $1-\sqrt{5}$.
Both of these values are less than $-1$, and eigenvalue interlacing implies that $K_{2,2,1}$ is not an
induced subgraph of $G'$.
Therefore $G'$ is a complete bipartite graph or $G'=K_{n'} { \sm } K_{m'}$ where $m'\leq n'-2$.
In the first case $G'$ has no eigenvalue $-1$, so $n-m-1=0$ which was excluded.
In the second case, the eigenvalue $-1$ has multiplicity $n-m-1$ in $G$ and multiplicity $n'-m'-1$ in $G'$,
so $n-m=n'-m'$.
Moreover, $G$ and $G'$ have the same number of edges, hence
$(n-m)(n-1)/2+m(n-m) = (n'-m')(n'-1)/2 + m'(n'-m')$.
Therefore $G = G' = K_n  \sm  K_m$.
\end{proof}
\\[5pt]
Note that, if $m = n-1$ the result need not be true.
Then $K_n  \sm  K_m = K_{1,n-1}$, and if $\ell$ divides $n-1$, then
$K_{1,n-1}$ is cospectral with $K_{\ell,k}+(n-\ell-k)K_1$, where $k=(n-1)/\ell$.

\section{The multiplicity of $-1$}

The complete graph $K_n$ has an eigenvalue $-1$ with multiplicity $n-1$.
If a few edges are deleted from $K_n$ then there will still be an eigenvalue $-1$ with
large multiplicity.
In this section we deal with graphs having the eigenvalue $-1$ with multiplicity at least $n-3$.
Clearly $K_n$ is the only graph for which the multiplicity of $-1$ is $n-1$.
\begin{proposition}
Let $G$ be a graph on $n$ vertices having an eigenvalue $-1$ with multiplicity $n-2$.
Then $G$ is the disjoint union of two complete graphs, and therefore $G$ is DS.
\end{proposition}
\begin{proof}
Suppose that $G$ has eigenvalue $-1$ with multiplicity $n-2$.
Then $A+I$ has rank 2.
Since $G\neq K_n$, we may assume that the first two rows of $A+I$ correspond to nonadjacent vertices.
Clearly these rows are independent and, since rank$(A+I)=2$, all rows of $A+I$ are linear
combination of the first two rows.
Then it follows straightforwardly that $\Gamma$ is the disjoint union of two complete graphs.
Clearly, the spectrum determines the order of each of the complete graphs.
\end{proof}
\\[5pt]
The following theorem is an unpublished result by Van~Dam, Haemers and Stevanovi\'c.
\begin{theorem}
Let $G$ be a graph with $n$ vertices having an eigenvalue $-1$ with multiplicity $n-3$.
Then $G = K_n \sm  K_{\ell,m}$, where $\ell,m \geq 1$, $\ell+m \leq n-1$,
or $G=K_k+K_\ell+K_m$, where $k,\ell,m \geq 1$, $k+\ell+m=n$.
\end{theorem}
\begin{proof}
In this case, $A+I$ is a symmetric matrix with rank~$3$ and we can assume that
\[
A+I=\left[\begin{array}{cc}
A_1 & X\\X^\top &  A_2
\end{array}\right],\ {\rm where}\ A_1 {\ \rm is \ a \ nonsingular \ } 3\times 3 {\rm \ matrix.}
\]
From ${\rm rank}(A+I) = {\rm rank\,}A_1$ it follows that $A_2=X^\top A_1^{-1}X$.
In particular, each column $\x$ of $X$ satisfies $\x^\top A_1^{-1} \x = 1$.
There are only two cases for which ${\rm rank\,}A_1 =3$: when $A_1=I$,
and when
\[
A_1=\left[\begin{array}{ccc}
1 & 1 & 1\\1 & 1 & 0\\1 & 0 & 1
\end{array}\right], {\rm\ in\ which\ case\ }
A_1^{-1} = \left[\begin{array}{rrr}
-1 & 1 & 1\\1 & 0 & -1\\1 & -1 & 0
\end{array}\right].
\]
In the first case the only possible columns of $X$ are the three unit vectors,
and in the second case the possible columns of $X$ are
$[\, 1\ 1\ 1\, ]^\top ,\ [\, 1\ 1\ 0\, ]^\top ,\ [\, 1\ 0\ 1\, ]^\top$.
With $A_2=X^\top A_1^{-1}X$ this leads to the following two possibilities for $A+I$:
\[
A+I=\left[\begin{array}{ccc}
J_k & O & O\\O & J_\ell & O\\O & O & J_m
\end{array}\right],\ {\rm or}\
A+I=\left[\begin{array}{ccc}
J_k & J & J\\J & J_\ell & O\\J & O & J_m
\end{array}\right].
\]
\\[-37pt]
\phantom{x}\end{proof}
\\[5pt]
\begin{corollary}\label{completebipartite}
$K_n  \sm  K_{\ell,m}$ is DS.
\end{corollary}
\begin{proof}
Put $G=K_n  \sm  K_{\ell,m}$.
If $\ell+m = n$ then $G = K_\ell+K_m$ which is DS.
Assume $k = n-\ell-m>0$, and suppose $G'$ is cospectral with $G$.
Then by the above theorem $G'$ is the disjoint union of three complete graphs,
%=K_k+K_\ell+K_m$,
or $G'=K_n \sm  K_{\ell',m'}$.
But the disjoint union of complete graphs is DS (see \cite{DH03}), therefore
$G'=K_n \sm  K_{\ell',m'}$.
Since $G$ and $G'$ have the same number of edges and triangles we find
$k\ell=k'\ell'$ and $k\ell(2n-k-\ell-2)/2=k'\ell'(2n-k'-\ell'-2)/2$,
hence $k=k'$ and $\ell=\ell'$.
\end{proof}
\\[5pt]
If the multiplicity of $-1$ is $n-4$, then there exist nonisomorphic cospectral graphs.
We'll present an example in the next section.

%%%%%%%%%%%%%%%%%%%%%%%%%%%%%%%%%%%%%%%%%%%
\section{Removing at most five edges}\label{sec: delete few}
%%%%%%%%%%%%%%%%%%%%%%%%%%%%%%%%%%%%%%%%%%%

Since we consider complements of graphs with few edges, and therefore few triangles the following relation
between the number of triangles of a graph and that of its complement is useful (see, for instance \cite{DH02}).
\begin{lemma}\label{lem: triangles in complement}
Let $G$ be a graph with $n$ vertices, $m$ edges, $t$ triangles, and degree sequence $d_1,d_2,\ldots,d_n$.
Let $\overline{t}$ be the number of triangles in the complement of $G$.
Then
$$
\overline{t}={n \choose 3}- (n-1)m + \frac{1}{2}\sum_{i=1}^n d_i^2-t \, .
$$
\end{lemma}
%
%The above equality follows from the inclusion-exclusion principle.
For closed walks of length 4 (for short: $4$-walks) things become more complicated.
\begin{lemma}\label{lem: 4-walks in the complement}
The number of $4$-walks in the complement of a graph $G$ only depends on the number of vertices and edges of $G$,
and the number of different subgraphs (not necessarily induced) in $G$ isomorphic to $P_3$, $K_2+K_2$, $P_4$ and $C_4$.
More precisely, if these numbers are $n$, $m$, $m_1$, $m_2$, $m_3$ and $m_4$, and $W_n=(n-1)^4+n-1$
is the number of $4$-walks in $K_n$, then the number of $4$-walks in the complement of $G$ equals
\[
W_n-(8n^2-32n+34)m+(8n-20)m_1+16m_2-8m_3+8m_4 \, .
\]
\end{lemma}
\begin{proof}
The result is a consequence of the inclusion-exclusion principle.
Assume $G$ and $K_n$ have the same vertex set.
Let $E$ be the edge set of $G$.
For a subset $F \subset E$, let $W_F$ denote the set of $4$-walks in $K_n$ containing all edges of $F$.
Then the total numbers of $4$-walks in $K_n$ that contain at least one edge from $E$ equals
%\begin{align*}
%|\bigcup_{e\in E^{\star}}W_e| & =\sum_{e\in E^{\star}} |W_e|-\sum_{(e_1,e_2)\in
%(E^{\star})^2}|W_{e_1}\cap W_{e_2}|+\sum_{(e_1,e_2,e_3)\in (E^{\star})^3}|W_{e_1}\cap W_{e_2}\cap W_{e_3}|\\&
%-\sum_{(e_1,e_2,e_3,e_4)\in (E^{\star})^4}|W_{e_1}\cap W_{e_2}\cap W_{e_3}\cap W_{e_4}|.
%\end{align*}
\[
|\bigcup_{|F|\geq 1} W_F| =
\sum_{|F|=1}|W_F| - \sum_{|F|=2}|W_F| + \sum_{|F|=3}|W_F| - \sum_{|F|=4}|W_F| \, .
\]
If $|F|=1$, then $|W_F|=8(n-2)(n-3)+8(n-2)+2$.
If $|F|=2$, then $|W_F|$ depends on the mutual position of the two edges.
If they have a vertex in common, then $|W_F| = 8(n-3)+4$, and if the two edges are
independent then $|W_F| = 16$.
Suppose $|F|=3$, then the three edges are a path in $G$, and there are $8$ different
$4$-walks in $K_n$ containing these edges.
Finally, if $|F|=4$, then the edges are a cycle of length $4$ in $G$.
Each of them leads to $8$ different 4-walks.
\end{proof}
\begin{theorem}\label{del5}
Let $G$ be a graph obtained by removing five or fewer edges from $K_n$, then $G$ is DS,
except if $G=K_7\sm(K_4 \sm  K_2)$ or $K_7\sm(K_{1,4}+K_2)$.
\end{theorem}
\begin{proof}
If we remove one, two or three edges the result follows from Proposition~\ref{match}, Theorem~\ref{complete} and
Corollary~\ref{completebipartite} with the exception of $K_n \sm  P_4$ and $K_n \sm  (P_3+K_2)$,
but these two graphs have different numbers of triangles, and therefore different spectrum.

In the case of 4 edges, we obtain eleven different graphs.
By use of Lemmas~\ref{lem: triangles in complement} and \ref{lem: 4-walks in the complement}
it follows straightforwardly that they can be distinguished by the number of triangles or $4$-walks.
Therefore the eleven graphs have different spectra.

There are 26 different graphs obtained from removing five edges from the complete graph.
Again we use Lemmas~\ref{lem: triangles in complement} and \ref{lem: 4-walks in the complement},
and find that just two pairs have the same number of triangles and $4$-walks.
These pairs are: $\{K_n\sm(K_4 \sm K_2)\, ,\, K_n\sm(K_{1,4}+K_2)\}$ and $\{K_n\sm P_6\, ,\, K_n\sm(C_4+K_2)\}$.
The graphs $K_7 \sm (K_4 \sm  K_2)$ and $K_7 \sm (K_{1,4}+K_2)$ are cospectral with spectrum $\{-1^3,-2,0,3\pm\sqrt{6}\}$.
For $n >7$, the adjacency matrix $A$ of $K_n\sm(K_4 \sm K_2)$ satisfies rank$(A+I)=n-4$,
whilst rank$(A'+I)=n-5$ for the adjacency matrix $A'$ of $K_n\sm(K_{1,4}+K_2)$.
Thus the two graphs have different multiplicities for the eigenvalue $-1$.
Similarly, $K_n\sm P_6$ and $K_n \sm (C_4+K_2)$ have different multiplicities for the eigenvalue $-1$ for all $n\geq 6$.
\end{proof}
\\[5pt]
We saw that the nonisomorphic cospectral graphs mentioned in the above theorem have spectrum
$\{-1^3,-2,0,3\pm\sqrt{6}\}$.
This pair is the promised example of graphs with an eigenvalue $-1$ of multiplicity $n-4$, which are not DS.
In the next section we shall see that for every $n\geq 7$ cospectral graphs exist that can be obtained
from $K_n$ by deleting $6$ edges.

%\section{The generalized spectrum}\label{generalized spectrum}

\section{\R-cospectral graphs}

Two graphs with adjacency matrices $A$ and $B$ are called {\em $\R$-cospectral} if $A+\alpha J$ is cospectral with
$B+\alpha J$ for every $\alpha\in\R$.
Clearly $\R$-cospectral graphs are cospectral ($\alpha=0$) and have cospectral complements ($\alpha=-1$).
Johnson and Newman~\cite{JN80} (see also \cite{DHK07}) proved the following theorem
(an orthogonal matrix $U$ is {\em regular} if $U J = J$).
\begin{theorem}\label{thm: graphs and complements cospectral}
For graphs $G$ and $H$ with adjacency matrices $A$ and $B$, respectively, the following are equivalent:
\begin{enumerate}[(a)]
\item $G$ and $H$ are $\R$-cospectral,
\item $G$ and $H$ are cospectral, and so are their complements,
%\item $A+\alpha J$ and $B+\alpha J$ are cospectral for all $\alpha\in\R$.
\item $ B =U^\top A U$ for some regular orthogonal matrix $U$.
\end{enumerate}
\end{theorem}
Clearly any pair of regular cospectral graphs is $\R$-cospectral.
Also graphs related by Godsil-McKay switching are $\R$-cospectral (see \cite{GM76}, \cite{DH03}).
Figure~1 presents a pair of nonisomorphic $\R$-cospectral graphs.
% with the smallest number of edges.
\begin{center}
\begin{tabular}{c}
\begin{tabular}{cc}
\setlength{\unitlength}{3pt}
\begin{picture}(26,26)(-4,-13)
\put(0,4.5){\circle*{2}} \put(0,-4.5){\circle*{2}}
\put(9,9){\circle*{2}} \put(9,-9){\circle*{2}}
\put(18,4.5){\circle*{2}} \put(18,-4.5){\circle*{2}}
\put(9,0){\circle*{2}} \put(0,4.5){\line(2,1){9}}
\put(0,-4.5){\line(2,-1){9}} \put(9,9){\line(2,-1){9}}
\put(9,-9){\line(2,1){9}} \put(0,-4.5){\line(0,1){9}}
\put(18,-4.5){\line(0,1){9}}
\end{picture}
 &\hspace{10pt}
\setlength{\unitlength}{3pt}
\begin{picture}(26,26)(-4,-10)
\put(-3,-6){\circle*{2}} \put(9,12){\circle*{2}}
\put(9,6){\circle*{2}} \put(3,-3){\circle*{2}}
\put(15,-3){\circle*{2}} \put(21,-6){\circle*{2}}
\put(9,0){\circle*{2}} \put(-3,-6){\line(2,1){12}}
\put(21,-6){\line(-2,1){12}} \put(9,0){\line(0,1){12}}
\end{picture}
\end{tabular}
\\
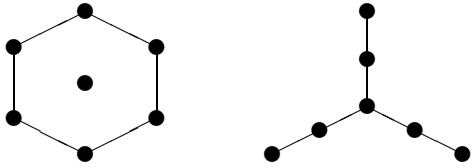
~~~Figure~1: Two $\R$-cospectral graphs
\end{tabular}
\end{center}
Consider graphs $G$ and $H$ with disjoint vertex sets.
The graph obtained from $G$ and $H$ by introducing all possible edges between a vertex of $G$ and a vertex of $H$
is called the {\em join} of $G$ and $H$, and denoted by $G \vee H$.
Thus, the complement of the join of $G$ and $H$ is the disjoint union of their complements.
\begin{theorem}
If $\{G_1,G_2\}$ and $\{H_1,H_2\}$ are two pairs of $\R$-cospectral graphs, then so are
$\{ G_1 \vee H_1 \, , \, G_2 \vee H_2 \}$ and $\{ G_1 + H_1 \, , \, G_2 + H_2 \}$.
\end{theorem}
\begin{proof}
Clearly it suffices to prove only the first claim.
Let $A_1$, $A_2$, $B_1$, $B_2$ be the adjacency matrices of $G_1$, $G_2$, $H_1$, $H_2$, respectively,
and let $U$ and $V$ the regular orthogonal matrices for which $U^\top A_1 U = A_2$ and $V^\top B_1 V = B_2$.
Then
\[
\left[\begin{array}{cc} U & O \\O & V \end{array}\right]^\top
\left[\begin{array}{cc} A_1 & J \\J & B_1 \end{array}\right]
\left[\begin{array}{cc} U & O \\O & V \end{array}\right] =
\left[\begin{array}{cc} A_2 & J \\J & B_2 \end{array}\right].
\]
Therefore $G_1 \vee H_1$ is $\R$-cospectral with $G_2 \vee H_2$.
\end{proof}
\\[5pt]
In particular, if we take $H_1=H_2=K_m$ we find the following.
\begin{corollary}\label{cor: cospectral with isolated vertices}
If $G_1$ and $G_2$ are $\R$-cospectral, then so are $K_n \sm G_1$ and $K_n \sm G_2$.
\end{corollary}
This result also follows from Lemma~2.8 of \cite{WX07}.
If $G_1=C_6+K_1$ and $G_2$ is the other graph of Figure~1, we find that
$K_n\sm C_6$ is $\R$-cospectral with $K_n\sm G_2$ for every $n\geq 7$.
Nonisomorphic $\R$-cospectral graphs exist for every number of edges at least six
(for example, the pair of Figure~1 remains $\R$-cospectral it in both graphs the endpoint of a
path $P_\ell$ is attached to a vertex of degree~2).
Therefore, for every $m\geq 6$ and every large enough $n$ one can obtain graphs which are not DS
by removing $m$ edges from $K_n$.

%%%%%%%%%%%%%%%%%%%%%%%%%%%%%%%%%%%%%%%%%%%
\section{Removing a path}
%%%%%%%%%%%%%%%%%%%%%%%%%%%%%%%%%%%%%%%%%%%

We expect that $K_n \sm P_\ell$ is DS for every $n\geq \ell\geq 2$.
By Theorem~\ref{del5} and \cite{DH02} it is true if $\ell\leq 6$
and if $n=\ell$.
Unfortunately we where not able to prove it in general.
The problem seems hard.
A reason for this may be that the complement $P_\ell+(n-\ell)K_1$
%of $K_n\sm P_\ell$
is not DS if $\ell$ is odd and $5\leq\ell\leq n-1$.
Indeed, if $m\geq 2$, then $P_{2m+1} + K_1$ is cospectral and nonisomorphic with $P_m + Y_{m+2}$,
where $Y_{m+2}$ is the graph obtained by attaching two pendent edges to one endpoint of $P_m$; see~\cite{CSS10}.
However, the disjoint union of two or more nontrivial paths is DS
(a nontrivial path has at least one edge); see~\cite{DH03,DH09}.
Therefore, the spectral characterization of the complement of the disjoint union of paths
becomes easier when the paths are nontrivial.
This will be addressed in a forthcoming paper.

However, we are able to prove that $K_n \sm P_\ell$ is determined by its generalized spectrum, which
means that there exists no graph which is $\R$-cospectral and nonisomorphic with $K_n \sm P_\ell$.
\begin{theorem}
The graph $K_n \sm P_\ell$ is determined by its generalized spectrum.
\end{theorem}
\begin{proof}
Let $G=P_\ell+(n-\ell)K_1$ (the complement of $K_n\sm P_\ell$).
Suppose $G'$ is cospectral with $G$.
The largest eigenvalue of $G$ (and $G'$) is strictly less than $2$.
Graphs with largest eigenvalue less than $2$ have been classified (see for example~\cite{CDS95}).
It is known that each component of such a graph is a path, or a path extended with one pendant edge
attached at a vertex of degree~$2$.
Let $x_i$ be the number of vertices in $G'$ with degree~$i$.
Then $x_i=0$ if $i\geq 4$.
We know that $G$ and $G'$ have the same number of edges and $4$-walks.
Therefore $\sum_i ix_i = 2\ell-2$, and $\sum_i i^2x_i = 4\ell-6$.
This gives $x_1-3x_3=2$, from which it follows that exactly one component of $G'$ is a nontrivial path.
So, if $G'$ is connected, then $G$ is isomorphic with $G'$.
Assume $G'$ is disconnected.
We decrease the number of components in $G'$ by removing the pendent edge in one component
(thus obtaining a second nontrivial path),
and inserting an edge between the endpoints of the two nontrivial paths.
This operation does not change the number of subgraphs $K_2+K_2$,
but it increases the number of subgraphs isomorphic to $P_4$.
After a number of such operations we obtain the graph $G$.
Now we apply Lemma~\ref{lem: 4-walks in the complement}.
With the notation of this lemma, the numbers $m$, $m_1$, $m_2$ and $m_4$ are the same for $G$ and $G'$.
But $m_3$ is different.
This implies that $G$ and $G'$ cannot have cospectral complements.~\end{proof}
\\[5pt]
{\large\bf Acknowledgement.}
We thank Andries E. Brouwer for pointing at the exception in Theorem~\ref{del5}.

\end{document}